\documentclass[11pt]{article}

\usepackage[english]{babel}
\usepackage{amsmath,amssymb,amsthm}
\usepackage{enumerate}
\usepackage[affil-it]{authblk}

\newtheorem{theorem}{Theorem}[section]
\newtheorem{proposition}[theorem]{Proposition}

\newtheorem{corollary}[theorem]{Corollary}

\theoremstyle{definition}

\newtheorem*{remark}{Remark}

\newcommand{\Z}{\mathbb{Z}}

\newcommand{\R}{\mathbb{R}}

\newcommand{\mx}{\mathbf{x}}
\newcommand{\mX}{\mathbf{X}}
\newcommand{\my}{\mathbf{y}}
\newcommand{\mF}{\mathbf{F}}
\newcommand{\mbU}{\mathbf{U}}
\newcommand{\mcU}{\mathcal{U}}
\newcommand{\mbu}{\mathbf{u}}
\newcommand{\mL}{\mathbf{L}}
\newcommand{\mC}{\mathcal{C}}
\newcommand{\oN}{\operatorname{N}}

\newcommand{\E}{\operatorname{{\bf E}}}

\renewcommand{\P}{\operatorname{\mathbb{P}}}
\newcommand{\PN}{\P_{\oN}}
\newcommand{\piN}{\pi_{\oN}}
\newcommand{\1}[1]{\operatorname{\mathbf{1}} \left \{ #1 \right \}}

\begin{document}

\title{Bak-Sneppen Backwards}

\author{Tom Alberts
\thanks{alberts@math.utah.edu}}
 \affil{University of Utah}

\author{Ga Yeong Lee
\thanks{glee2@caltech.edu}}
\affil{California Institute of Technology}

\author{Mackenzie Simper
\thanks{mackenzie.simper@utah.edu}}
\affil{University of Utah}

\date{}

\maketitle

\begin{abstract}
We study the backwards Markov chain for the Bak-Sneppen model of biological evolution and derive its corresponding reversibility equations. We show that, in contrast to the forwards Markov chain, the dynamics of the backwards chain explicitly involve the stationary distribution of the model, and from this we derive a functional equation that the stationary distribution must satisfy. We use this functional equation to derive differential equations for the stationary distribution of Bak-Sneppen models in which all but one or all but two of the fitnesses are replaced at each step. This gives a unified way of deriving Schlemm's expressions for the stationary distributions of the isotropic four-species model, the isotropic five-species model, and the anisotropic three-species model.
\end{abstract}

\section{Introduction \label{sec:intro}}

Since its first introduction in \cite{Bak_Sneppen:original}, the Bak-Sneppen model and its many novel features have become of great interest across a wide variety of disciplines for a wide variety of reasons. To biologists Bak-Sneppen is a simplified model of biological evolution that incorporates mechanisms for natural selection with spatial interaction, and has had some success \cite{bose_chaudhuri} in explaining experimental results on the fitnesses of bacteria in an evolutionary experiment \cite{Lenski_Travisano}. To physicists Bak-Sneppen is a model that exhibits so-called \textit{self-organized criticality} \cite{Bak_Tang_Wiesenfeld:SOC, jensen1998self}, the notion that a system without tunable parameters can still exhibit the typical properties of a phase transition at a critical point. From the point of view of probabilists, Bak-Sneppen is a discrete-time Markov chain on a continuous state space that is simple to describe but notoriously difficult to analyze. The dynamics of the Markov chain are as follows: finitely many species are arranged on a circle and each assigned a fitness value, initially chosen to be independent and uniformly distributed between $0$ and $1$. At each time the species with minimum fitness is eliminated and replaced with a new species with an independent fitness value that is uniform on $[0,1]$; this is the natural selection component of the model. In addition, the two species' adjacent to the species with minimum fitness are eliminated and replaced with new species with independent fitness values uniform on $[0,1]$; this is the spatial interaction component of the model. It is well known that under these dynamics there is a unique stationary distribution for the fitness values, although it generally seems to be a complicated function with no closed-form expression.

It is widely believed, however, that as the number of species in the model grows infinitely large the stationary distribution starts to become simpler. The longstanding but still unproven conjecture is that in the limit the individual species' fitnesses become uniformly distributed on an interval $(f_c, 1)$, where $f_c$ is a value somewhere close to $2/3$, and moreover that the fitnesses all become independent of each other. The latter is particularly surprising and in marked contrast to the situation for a finite number of species where the limiting fitnesses are dependent and supported on all of $(0,1)$. Strong numerical evidence for the limiting form of the one-dimensional marginals has been provided by simulations in \cite{Grassberger, jensen1998self, bak:nature}, and significant mathematical progress towards this conjecture \cite{meester_znam:limit_behavior, meester_znam:critical_thresholds, GMN:avalanche_bounds, GMV:maximal_avalanches} has been made using the notion of \textit{avalanches}. More recent papers have studied the Bak-Sneppen model using the framework of \textit{rank-driven Markov processes} \cite{GKW:BS_rank_driven, GKW:rank_driven_JSP}.

A previously unexplored aspect of the Bak-Sneppen model is its time reversed process. Time reversal is a very powerful tool in Markov chain theory: if one starts the Markov chain in stationarity then the process run backwards is also a time-homogeneous Markov chain with the same stationary distribution yet different dynamics. For the Bak-Sneppen model these reverse dynamics are particularly interesting. Both the forward chain and the reverse chain follow the same procedure of replacing three species' fitnesses at each time step, though the mechanism for how this occurs is very different for each direction. In the forwards direction of Bak-Sneppen the chosen species is entirely deterministic, given the current configuration of fitnesses, but the replacement fitness values are entirely independent of the configuration. In particular this means we have no a priori knowledge of what the new fitnesses will be. In the backwards direction of Bak-Sneppen this turns out to be entirely the opposite. The species that will be replaced is chosen randomly, with probabilities determined by the current configuration. Furthermore, the replacement fitnesses are highly influenced by the current configuration and the chosen species. In particular, the fitness value of the middle species must be chosen so that it is the minimum of the configuration at the next step.

In this paper we derive formulas for the transition rules for the reversed process. As is to be expected, they depend explicitly on the stationary distribution for the Markov chain which is in general unknown. However, we are able to use the reverse dynamics to derive a functional equation that the stationary distribution must satisfy, which provides a new tool for deriving its properties. It is not yet fully clear how effective this tool will be for analyzing the stationary distribution for large numbers of species, but one of the main results of this paper is that it is very effective for small numbers of species.

Analysis of the Bak-Sneppen model with small numbers of species was recently considered by Schlemm \cite{schlemm:four_species, schlemm:five_species}, who was able to derive exact expressions for the stationary distribution in the four and five species cases. He does this by applying the so-called power method or von Mises iteration: he repeatedly updates the fitness distribution with successive application of the transition kernel and studies the convergence of this sequence of distributions. He shows that when the initial distribution is uniform the subsequent distributions are all polynomial expressions in the fitness values (that obey certain symmetry properties), and in the limit these polynomials converge to a Taylor series which can be summed exactly. This involves setting up complicatd recursive equations for the coefficients of the polynomials and finding solutions for each individual recursion. In the four species case the solution Schlemm produces \cite{schlemm:four_species} is very explicit, and he uses the same method to find the stationary distribution for the anisotropic three-species model, in which only one of the neighboring species has its fitness parameter replaced. In the five species case \cite{schlemm:five_species} the stationary distribution is expressed as the solution to a certain system of differential equations that has no simple closed form expression.

Using the reversed process and the resulting functional equations, we are able to re-derive all of Schlemm's formulas in a shorter and more unified way. Our method does not require setting up a large collection of recursive equations and proving their convergence properties. Instead we only have to solve certain integral equations which, with modest effort, can be converted into differential equations in a small number of variables. In the isotropic four species and the anisotropic three species cases this leads to an easily solved ordinary differential equation that has a very explicit solution, see Theorem \ref{thm:frsi_models} for the exact expression. In the isotropic five species case we are able to reduce the integral equations into a system of three ordinary differential equations, whose solutions, although not explicit, can be combined to produce the stationary distribution for that model. See Theorem \ref{thm:two_non_replace} for the full statement of the solution.

Our approach has an added benefit in that it generalizes beyond just having four or five species in the model, so long as the number of \textit{non}-replaced species is kept at one or two. Both Theorems \ref{thm:frsi_models} and \ref{thm:two_non_replace} allow for a more general setup of the Bak-Sneppen model in which the set of non-replaced species at each step is kept small, so long as it occupies the same relative position with respect to the species with minimum fitness, but the set of replaced species can be made arbitrarily large. As our proofs show this extra feature requires virtually no change in the analysis.

The outline of this paper is as follows: in Section \ref{sec:reversibility} we set our notation for the model and then derive its reversibility equations in  Theorem \ref{thm:BS_adjoint}. Using this we describe the dynamics of the reverse Markov chain in Corollary \ref{cor:reverse_joint_dist}, and derive a functional equation that the stationary distribution must satisfy in Proposition \ref{prop:stationary_dist_functional_eqn}. In Section \ref{sec:small_species} we use this functional equation to re-derive Schlemm's exact expressions for the stationary distributions of the isotropic four-species model, the isotropic five-species model, and the anisotropic three-species model, along with our generalizations to the situation where the number of species is arbitrarily large but the number of non-replaced species is either one or two.
\newline \newline
\textbf{Acknowledgments: } Tom Alberts thanks Siva Athreya and Eric Cator for helpful discussions, and the support of the Scott Robert Johnson fellowship at the California Institute of Technology. Ga Yeong Lee thanks the Caltech Summer Undergraduate Research Fellowships (SURF) program and Margaret Leighton for their support of her research. Mackenzie Simper thanks the internal REU program at the University of Utah's Mathematics department for support of her research.

\section{The Bak-Sneppen Model and its Reverse Dynamics \label{sec:reversibility}}

\subsection{Definition of the Model and Notation}

The underlying dynamics are very simple to describe. The model begins with $N$ species arranged in a circle and to each one is attached a fitness parameter taking values between zero and one. Typically the initial choice is to make them iid uniform random variables. At each iteration the species with minimum fitness is selected, and its fitness \textit{along with that of some of its neighbors} is replaced with new and independently chosen uniform random fitnesses. The selection of the minimum fitness is a way to model natural selection while the rule that the neighbors also have their fitnesses replaced is a simplified model of spatial interaction between the species.

We describe the Bak-Sneppen model using the language of Markov chains. For an integer $N \geq 4$ we label the different species of the population by $\Z_N = \Z / (N \Z)$, the group of integers modulo $N$, and it is understood that all operations on the species' labels are performed modulo $N$. The Markov chain takes values in the hypercube $[0,1]^{\Z_N}$. For shorthand we will use the symbol $\mC_N$ to denote the hypercube. Generic elements of this space we write as $\mx = (\mx_1, \mx_2, \ldots, \mx_N)$. We also write $\mF_0, \mF_1, \mF_2, \ldots \in \mC_N$ for the random sequence of fitness configurations that form the Markov chain. Note that for bold lowercase letters the subscript indicates the particular species within the fitness configuration, while for bold uppercase letters the subscript denotes the time in the Markov chain. We denote the uniform measure on the hypercube by $\mcU_N$, and for the forward Markov chain it is standard to assume that $\mF_0 \sim \mcU_N$. When studying the backwards Markov chain we will typically make the assumption that $\mF_0$ starts from the stationary distribution. We then introduce sets $E_i$ that will allow us to describe the dynamics of the Bak-Sneppen Markov chain in wide generality, including both the isotropic and anisotropic cases. For each $i \in \Z_N$ let $E_i$ be the set of species whose fitnesses will be replaced if $i$ is the species with minimum fitness . In the isotropic case $E_i = \{i-1, i, i+1 \}$ and in the anisotropic case $E_i = \{ i, i+1 \}$, but many of our later arguments apply to arbitrary choices of the $E_i$. For a given choice we let $\mL_i$ be the projection operator onto the linear subspace $\R^{E_i}$, i.e. $\mL_i$ takes a vector in $\R^{\Z_N}$ and sets to zero all the entries corresponding to indices that are not in $E_i$. We also write $\mL_i^c$ for the projection onto $\R^{E_i^c}$; note then that $\mathbf{I} = \mL_i + \mL_i^c$ where $\mathbf{I}$ is the identity operator. By defining $a(\mx) = \operatorname{argmin} \mx$ (which is well-defined Lebesgue a.e.), the Bak-Sneppen dynamics can be written as
\begin{align}\label{eqn:BS_forward_dynamics}
\mF_{k+1} = \mF_{k} + \mL_{a(\mF_k)}(\mbU_k - \mF_k),
\end{align}
where $\mbU_0, \mbU_1, \ldots$ are iid uniform random variables in the hypercube. The second term on the right hand side replaces the fitness parameters determined by the minimum fitness with new and independent uniform random variables. For functions $f : \mC_N \to \R$ this Markov chain is equivalently described by the Markov operator
\begin{align}
Pf(\mx) := \E \left[ \left. f \left( \mF_{k+1} \right) \right| \mF_k = \mx \right] &= \E_{\mcU_n} \! \left [ f \left( \mx + \mL_{a(\mx)}(\mbU - \mx) \right) \right] \notag \\
&= \int_{\mC_N} \!\! f \! \left( \mx + \mL_{a(\mx)}(\mathbf{u} - \mx)  \right) \, d \mathbf{u}. \label{eqn:BS_forward_operator}
\end{align}
In the next section we will determine the adjoint of this operator on the function space $L^2(\mC_N, \PN)$, where $\PN$ is the stationary measure for this chain. The adjoint is the Markov operator corresponding to the backwards chain and from it we can deduce the reverse dynamics. Before proceeding with this we collect some useful facts about the stationary distribution, all of which are proved in other works.

\begin{theorem}
There exists a unique probability measure $\PN$ on $\mC_N$ that is stationary for the Bak-Sneppen model. Moreover $\PN$ is absolutely continuous with respect to Lebesgue measure.
\end{theorem}

\begin{proof}
Existence and uniqueness is proved in \cite{gillet:thesis}, based on techniques from \cite{meyn_tweedie:MC}. That proof holds on general graphs and therefore applies to our situation of arbitrary sets $E_i$. A similar proof is found in \cite{schlemm:four_species} for the typical model with $E_i = \{i-1, i, i+1 \}$ and also includes a proof of the absolute continuity, which can easily be extended to handle the case of general sets $E_i$.
\end{proof}

In the remainder of the paper we will write
\begin{align*}
d \PN(\mx) = \piN(\mx) \, d \mx.
\end{align*}
Implicitly this depends on the choice of the sets $E_i$, but this will always be clear from the context.

\subsection{The Reverse Dynamics}

The reversible dynamics describe the evolution of the process $\mF_{T-k}, 0 \leq k \leq T$, where $T$ is a large positive integer and $\mF_0$ is assumed to be distributed according to $\PN$. With this distribution for $\mF_0$ the reverse chain is also a time-homogeneous Markov process, but with different dynamics governing its evolution. In this section we determine exactly what the dynamics are for the reverse chain. Their main feature is that they explicitly involve the unknown stationary distribution $\PN$, in contrast to the forward dynamics which make no reference to it.

To determine the reversible dynamics for the Bak-Sneppen model we compute the adjoint of the Markov operator $P$. Henceforth we will call this adjoint $Q$, and we recall that it satisfies
\begin{align}
\langle f, Pg \rangle_{\PN} = \langle Qf, g \rangle_{\PN}
\end{align}
for all functions $f, g \in L^2(\mC_N, \PN)$. The inner product is the standard $L^2$ one
\begin{align*}
\langle f, g \rangle_{\PN} = \E_{\PN} \left[ f(\mX) g(\mX) \right] = \int_{\mC_N} f(\mx) g(\mx) \, d \PN(\mx).
\end{align*}
The equivalent description of $Q$ is as the Markov operator for the reverse chain, that is
\begin{align*}
Qf(\mx) = \E \left[ \left. f \left( \mF_{T-k-1} \right) \right| \mF_{T-k} = \mx \right],
\end{align*}
so long as $\mF_0 \sim \PN$.

\begin{theorem}\label{thm:BS_adjoint}
The adjoint operator $Q$ acts on $L^2(\mC_N, \PN)$ via the formula
\begin{align}
Qf(\mx) = \frac{1}{\piN(\mx)} \sum_{i \in \Z_N} H_i f (\mx),
\end{align}
where the operators $H_i$ are defined by
\begin{align*}
H_i f(\mx) = \int_{\mC_N} f(\mL_i \my + \mL_i^c \mx) \1{a(\mL_i \my + \mL_i^c \mx) = i} \piN(\mL_i \my + \mL_i^c \mx) \, d \my.
\end{align*}
\end{theorem}

\begin{remark}
The operators $H_i$ are essentially expectations with respect to the marginal distribution of $\mL_i^c \mX$, where $\mX$ is distributed according to $\P_N$. The indicator function term only integrates over vectors $\my$ which put the minimum of $\mL_i \my + \mL_i^c \mx$ at species $i$. Note that the right hand side of the equation depends only on $\mL_i^c \mx$, hence if $\mL_i^c \mx_1 = \mL_i^c \mx_2$ then $H_i f(\mx_1) = H_i f(\mx_2)$.
\end{remark}

\begin{proof}
In this proof we use the notation $\mx_i = \mL_i \mx$, $\mx_i^c = \mL_i^c \mx$ for the projections of a vector $\mx \in \mC_N$, and for a set $A \subset \Z_N$ of species we write $\mC(A) = [0,1]^A$. Now let $f, g \in L^2(\mC_N, \PN)$. Then by definition of $Pg$ and the inner product we have
\begin{align*}
\langle f, Pg \rangle_{\PN} &= \int_{\mC_N} f(\mx) Pg(\mx) \piN(\mx) \, d \mx \\
&= \int_{\mC_N} \int_{\mC_N} f(\mx) g(\mx + \mL_{a(\mx)}(\mbu - \mx)) \piN(\mx) \, d \mbu \, d \mx \\
&= \sum_{i \in \Z_N} \int_{\mC_N} \int_{\mC_N}  f(\mx) g(\mx + \mL_{i}(\mbu - \mx)) \1{a(\mx) = i} \piN(\mx) \, d \mbu \, d \mx \\
&= \sum_{i \in \Z_N} \int_{\mC(E_i^c)} \int_{\mC(E_i)} \int_{\mC(E_i)} \!\!\!\! f(\mx_i + \mx_i^c) g(\mx_i^c + \mbu_i) \1{a(\mx_i + \mx_i^c) = i} \piN(\mx_i + \mx_i^c) \, d \mbu_i \, d \mx_i \, d \mx_i^c \\
&= \sum_{i \in \Z_N} \int_{\mC(E_i^c)} \int_{\mC(E_i)} \!\! g(\mx_i^c + \mbu_i) \int_{\mC(E_i)} \!\!\!\! f(\mx_i + \mx_i^c) \1{a(\mx_i + \mx_i^c) = i} \piN(\mx_i + \mx_i^c) \, d \mx_i \, d \mbu_i \, d \mx_i^c \\
&= \sum_{i \in \Z_N} \int_{\mC(E_i^c)} \int_{\mC(E_i)} \!\! g(\mx_i^c + \mbu_i) H_i f(\mx_i^c) \, d \mbu_i \, d \mx_i^c \\
&= \sum_{i \in \Z_N} \int_{\mC_N} g(\my) H_i f(\my - \mL_i y) \, d \my \\
&= \sum_{i \in \Z_N} \int_{\mC_N} g(\my) \frac{H_i f(\my)}{\piN(\my)} \piN(\my) \, d \my.
\end{align*}
The second last equality follows by writing $\my = \mx_i^c + \mbu_i$ and noting that for Lebesgue measure $d \my = d \mbu_i \, d \mx_i^c$. By definition of $\my$ we have $\mL_i \my = \mbu_i$, and by the remark before the proof we also have that $H_if(\my - \mL_i \my) = H_if(\my)$.
\end{proof}

From the adjoint $Q$ we can derive a description for the dynamics of the backwards Markov chain. As we mentioned in the introduction, the backwards Markov chain chooses a species $i$ at random and then replaces the fitness parameters of the species in the set $E_i$ with randomly chosen fitnesses. Therefore the backwards Markov chain is fully described by the joint distribution of the chosen species and the replaced fitnesses. For the forwards Markov chain this joint distribution only depends on the location of the species with the minimum fitness, while for the backwards Markov chain it depends much more heavily on the current value of all the fitness parameters.

\begin{corollary}\label{cor:reverse_joint_dist}
Suppose $\mF_0 \sim \PN$. Conditionally on $\mF_{T-k} = \mx$, the joint probability of the chosen species and the replaced fitnesses for $\mF_{T-k-1}$ is
\begin{align*}
\1{a(\mL_i \my + \mL_i^c \mx) = i} \, \frac{\piN(\mL_i \my + \mL_i^c \mx)}{\piN(\mx)} \, d \my.
\end{align*}
In particular, the probability that species $i$ is selected is
\begin{align*}
\frac{1}{\piN(\mx)} \int_{\mC_N} \1{a(\mL_i \my + \mL_i^c \mx) = i} \piN(\mL_i \my + \mL_i^c \mx) \, d \my.
\end{align*}
\end{corollary}

\begin{proof}
This is essentially a restatement of Theorem \ref{thm:BS_adjoint}. For the second part one can simply integrate out the $\my$ variable of the joint density, or use Theorem \ref{thm:BS_adjoint} on the function $f(\mx) = \1{a(\mx) = i}$. This is because for the backwards chain the chosen species is the one that is assigned the minimum fitness value.
\end{proof}

\subsection{A Functional Equation for the Stationary Distribution}

\begin{proposition}\label{prop:stationary_dist_functional_eqn}
The stationary distribution $\piN$ satisfies the functional equation
\begin{align}\label{eqn:stationary_dist_functional_eqn}
\piN(\mx) = \sum_{i \in \Z_N} \int_{\mC_N} \1{a(\mL_i \my + \mL_i^c \mx) = i} \piN(\mL_i \my + \mL_i^c \mx) \, d \my.
\end{align}
\end{proposition}

\begin{proof}
Consider the constant function $f : \mC_N \to \R$ given by $f(\mx) = 1$. Since $Q$ is the Markov operator of the reverse chain we automatically have $Qf(\mx) = 1$ for all $\mx$. The result then follows from Theorem \ref{thm:BS_adjoint}.
\end{proof}

\begin{remark}
Proposition \ref{prop:stationary_dist_functional_eqn} can also be derived from the standard adjoint equation. Indeed if $P^*$ is the adjoint of $P$ in the space $L^2(\mC_N, d \mx)$, then $\piN$ is the unique solution to $P^* \piN = \piN$ that is also a probability distribution. Calculations similar to those in the proof of Theorem \ref{thm:BS_adjoint} give us that
\begin{align*}
P^*f(\mx) = \sum_{i \in \Z_N} \int_{\mC_N} f(\mL_i \my + \mL_i^c \mx) \1{a(\mL_i \my + \mL_i^c \mx) = i} \, d \my,
\end{align*}
from which one sees that $P^* \piN = \piN$ is equivalent to \eqref{eqn:stationary_dist_functional_eqn}. We have chosen to make our derivation this way since it also gives us a description of the reverse dynamics.
\end{remark}

\begin{remark}
Each term on the right hand side of \eqref{eqn:stationary_dist_functional_eqn} is a function of $\mL_i^c \mx$ since the $\my$ variables are integrated out. For ``shift-invariant'' models where $E_i = i + E_0$ for all $i$, which includes the standard isotropic and anisotropic Bak-Sneppen models, symmetry arguments show that it is the same function for each $i$. That is, for each fixed $N$ there exists a function $q_N : [0,1]^{E_0^c} \to \R_+$ such that
\begin{align*}
\piN(\mx) = \sum_{i \in \Z_N} q_N(\mL_i^c \mathcal{S}^{-i} \mx),
\end{align*}
where $\mathcal{S}$ is the shift operator that cyclically increments the label of each species. This decomposition is similar to the one found in \cite[Proposition 2]{schlemm:four_species}, and we will use it for determining exact solutions for small numbers of species.
\end{remark}

\section{Exact Solutions for Small Numbers of Non-Replaced Species \label{sec:small_species}}

In this section we show that for small numbers of non-replaced species, that is when $E_0$ is almost all of $\Z_N$, equation \eqref{eqn:stationary_dist_functional_eqn} can be used to derive an explicit expression for $\piN$.

\subsection{Shift Invariant Models for One Non-Replaced Species}

Here we consider the forwards Bak-Sneppen model with $N$ species in which, at each time, all but one of the species have their fitnesses replaced, including the species with minimum fitness. Note that this includes the isotropic four-species model and the anisotropic three-species model studied by Schlemm. We also assume that the model is shift invariant, in the sense that $E_i = i + E_0$ for each $i$. In words this means that the set of species replaced at each time is a re-centering of a given fixed set around the species with minimum fitness.

\begin{theorem}\label{thm:frsi_models}
Suppose that $E_i = i + E_0$ for each $i$, and that $E_0 = \Z_N \backslash \{ k \}$ for some $k \neq 0$. Then the stationary distribution $\piN(\mx)$ is
\begin{align*}
\piN(\mx) = \frac{(N-1)(N-2)}{N} \sum_{i \in \Z_N} \frac{1 - (1-\mx_i)^{N-1}}{\left( (1 - \mx_i)^{N-1} + N-2 \right)^2}.
\end{align*}
\end{theorem}

\begin{proof}
By the second remark after Proposition \ref{prop:stationary_dist_functional_eqn}, the full-replacement and shift invariance assumptions give that $\piN$ has the form
\begin{align*}
\piN(\mx) = \sum_{i \in \Z_N} q_N (\mx_i)
\end{align*}
for a to-be-determined function $q_N : [0,1] \to \R_+$. By plugging this form into equation \eqref{eqn:stationary_dist_functional_eqn} and collecting all of the individual fitness variables $\mx_i$ into individual equations we get a collection of functional equations that $q_N$ must satisfy. By the shift invariance assumption it is the same functional equation for each fitness variable, and so we may assume that $i=0$. Using that $\mL_0^c \mx = \mx_k$ we have
\begin{align*}
q_N(\mx_k) = \int_{\mC_N} \1{a(\mL_0 \my + \mL_0^c \mx) = 0} \left( q_N(\mx_k) + \!\!\!\! \sum_{j \in \Z_n \backslash \{ k \}} \!\!\!\! q_N(\my_j) \right ) \, d \my.
\end{align*}
The indicator function constrains $\my$ to take on values where the minimum of $\mL_0 \my + \mL_0^c \mx$ is at species zero, with the fitness value $\mx_k$ given and remaining fixed. Therefore the constraint on the integration is equivalent to $\my_0 \leq \mx_k$, and $\my_j \geq \my_0$ for all $j \neq 0,k$. Writing from now on $\mx_k = x$, the latter equation becomes
\begin{align*}
q_N(x) = \int_{0}^{x} \int_{\mC_N} \left( q_N(x) + \!\!\!\! \sum_{j \in \Z_n \backslash \{ k \}} \!\!\!\! q_N(\my_j) \right ) \left( \prod_{j \neq 0, k} \1{\my_j \geq \my_0} d \my_j \right)  \, d \my_0.
\end{align*}
Now introduce the notation $Q_N(x) = \int_0^x q_N(t) \, dt$. With this the inner integral is simple to evaluate, although one should note that in the summation term the case $j=0$  requires special care. This leaves us with
\begin{align*}
q_N(x) = \int_0^x (1 - \my_0)^{N-2} (q_N(x) + q_N(\my_0)) + (N-2)(1 - \my_0)^{N-3} (Q_N(1) - Q_N(\my_0)) \, d \my_0.
\end{align*}
Straightforward calculations give
\begin{align*}
q_N(x) = \frac{1}{N-1} ( 1 - (1-x)^{N-1} ) q_N(x) + ( 1 - (1-x)^{N-2} ) Q_N(1) + (1-x)^{N-2}Q_N(x).
\end{align*}
This is an ordinary differential equation for $Q_N$ with boundary conditions $Q_N(0) = 0$ and $Q_N(1) = 1/N$, the latter following from the constraint that $\piN$ must integrate to $1$. The ODE is easily solved via the product rule to give
\begin{align*}
Q_N(x) = \frac{(N-1)x + (1-x)^{N-1} - 1}{N((1-x)^{N-1} + N-2)},
\end{align*}
which we then differentiate to get
\begin{align*}
q_N(x) = \frac{(N-1)(N-2)}{N} \frac{1 - (1-x)^{N-1}}{((1-x)^{N-1} + N-2)^2}.
\end{align*}
\end{proof}

\begin{remark}
For $N=4$ simple algebra gives
\begin{align}
q_4(x) = \frac{3 \cdot 2}{4} \frac{1 - (1-x)^3}{((1-x)^3 + 2)^2} = \frac{3}{2} \frac{x(3-x(3-x))}{(3-x(3-x(3-x)))^2}, \label{eqn: four_species_exact}
\end{align}
which agrees with the solution of Schlemm \cite{schlemm:four_species} for the isotropic four-species model. Similarly, for $N=3$ our formula reduces to
\begin{align*}
q_3(x) = \frac{2 \cdot 1}{3} \frac{1 - (1-x)^2}{((1-x)^2 + 1)^2} = \frac{2}{3} \frac{x(2-x)}{(2-x(2-x))^2},
\end{align*}
which agrees with formula of Schlemm for the anisotropic three-species model. As $N \rightarrow \infty$, it is straightforward to see that
\begin{align*}
\lim_{N \rightarrow \infty} N q_N (x) = 1
\end{align*}
pointwise in $x$, and even uniformly on any sub-compact set that is bounded away from 0. This agrees with the intuition that when all but one species is replaced at each time step, the asymptotic distribution of the fitnesses is still uniform on $[0, 1]$.
\end{remark}

\subsection{Shift Invariant Models for Two Adjacent Non-Replaced Species}

We again consider the shift invariant case $E_i = i + E_0$, but we now assume that $E_0$ consists of all but two locations which are adjacent to each other, and neither of which is adjacent to species zero. This assumption contains the standard, isotropic five-species model. In this framework we are able to re-derive Schlemm's recent \cite{schlemm:five_species} description of the asymptotic fitness distribution for the isotropic, $N=5$ model as the solution to a particular system of differential equations.

\begin{theorem}\label{thm:two_non_replace}
Suppose that $E_i = i + E_0$ for each $i$, and that $E_0 = \Z_N \backslash \{k, k+1\}$ where $k \not \in \{-2,-1,0,1 \}$. Then the stationary distribution is
\begin{align*}
\piN(\mx) = \sum_{i \in \Z_N} q_N(\mx_i, \mx_{i+1})
\end{align*}
where $q_N : \{ (u,v) : 0 \leq v \leq u \leq 1 \} \to \R_+$ has the form $q_N(u,v) = B_N(u)G_N(v) + A_N(v)$, with $A_N, B_N,$ and $G_N$ solutions to the system of equations
\begin{align}
(1- a_{N-2}(x))G_N'(x) - a_{N-2}'(x)G_N(x) + a_{N-3}'(x) \int_0^x G_N(s) \, ds = 0, \label{eqn:DEforG}
\end{align}
\begin{align}
& A_N'(x) = \frac{1}{1-x} \left[ - G_N'(x) \int_x^1 B_N(s) \, ds - B_N'(x) \int_0^x G_N(s) ds \right], \label{eqn:Afunction}
\end{align}
and
\begin{align}
A_N(v) &= a_{N-2}(v) A_N(v) + a_{N-3}(v)G_N(v) \int_v^1 B_N(s) \,ds + a_{N-3}(v)(1-v)A_N(v) \notag \\
&\,\, + 2 \int_0^v a_{N-3}(t) A_N(t) \, dt - a_{N-3}(v) \int_0^v A_N(t) \, dt \notag \\
&\,\, + B_N(v) \int_0^v a_{N-3}(t) G_N(t) \, dt \label{eqn:thirdEqn} \\
&\,\, + 2 \int_0^v (1-t)^{N-4} G_N(t) \int_t^1 B(s) \, ds \, dt + 2 \int_0^v (1-t)^{N-3} A_N(t) \, dt \notag \\
&\,\, + (N-5)\int_0^v (1-r)^{N-5} \int_r^1 \int_r^s \left[ B_N(s)G_N(t) + A_N(t) \right] \, dt \, ds \, dr \notag \\
&\,\, + (N-5)\int_0^v (1-r)^{N-5} \int_r^1 \int_t^1 \left[ B_N(s)G_N(t) + A_N(t) \right] \, ds \, dt \, dr, \notag
\end{align}
where $a_k(v)$ is the function
\begin{align*}
a_k(v) = \int_0^v (1-t)^{k-1} \, dt = \frac{1}{k} \left[ 1 - (1-v)^k \right].
\end{align*}
\end{theorem}

\begin{remark}
Note that equation \eqref{eqn:DEforG} for $G_N$ does not involve $A_N$ or $B_N$, and therefore can be solved on its own as a second-order differential equation. Then given the solution for $G_N$, the two equations \eqref{eqn:Afunction}, \eqref{eqn:thirdEqn} form a system which can be solved for $B_N$ and $A_N$. In the $N=5$ case this system appears in \cite[Corollary 2 and Proposition 5]{schlemm:five_species}. Furthermore, as in \cite[Proof of Theorem 1]{schlemm:five_species} equation \eqref{eqn:thirdEqn} can be differentiated multiple times to become a differential equation solely for $B_N$, which together with the solution for $G_N$ and equation \eqref{eqn:Afunction} determines $A_N$. Note that in equation \eqref{eqn:Afunction} the terms involving $A_N$ are entirely separate from the terms involving $B_N$ and $G_N$.
\end{remark}

\begin{remark}
In the $N=5$ case our solutions appear to be slightly different from \cite[Theorem 1]{schlemm:five_species}, but can be made to agree by making the substitutions
\begin{align*}
\mathcal{B}_0(x) = \int_0^{1-x} A(s) \, ds, \quad \mathcal{B}_1 (x) = \int_{1-x}^1 B(s) \, ds, \text{    and  } \mathcal{G}(x) = -\int_0^{1-x} G(s) \, ds.
\end{align*}
With this substitution, \eqref{eqn:DEforG} becomes
\begin{align*}
\left(1 - \frac{(1 - (1-x)^3))}{3} \right) \mathcal{G}''(1-x) - (1-x)^2 \mathcal{G}'(1-x) + (1-x)\mathcal{G}(1-x) = 0,
\end{align*}
which is equivalent to
\begin{align*}
(z^3 + 2) \mathcal{G}''(z) + 3z^2 \mathcal{G}'(z) + 3z \mathcal{G}(z) = 0.
\end{align*}
This same equation can be found towards the end of \cite[Proof of Theorem 1]{schlemm:five_species}, however we arrive at it very differently and now present an alternative method for solving it. In the latter formula we allow $z$ to be complex, which makes the differential equation second order with regular singular points at the three roots of $z^3 + 2 = 0$. Therefore the ODE can be reduced to the standard hypergeometric differential equation and its solution expressed in terms of hypergeometric functions or Riemann's P-function; see \cite{AbSteg} or \cite{poole} for more details. In this particular case the correct reduction is by allowing $\mathcal{G}(z) = \mathcal{H}(-z^3/2)$, which leads to the hypergeometric equation
\begin{align*}
3z(1-z) \mathcal{H}''(z) + (2 - 5z) \mathcal{H}'(z) - \mathcal{H}(z) = 0.
\end{align*}
This is the standard form of Euler's hypergeometric equation, whose two linearly independent solutions are expressed in terms of Gauss' hypergeometric function $\,_2F_1$ as
\begin{align*}
\,_2F_1 \left( \frac{1}{3} + i \frac{\sqrt{2}}{3}, \frac{1}{3} - i \frac{\sqrt{2}}{3}; \frac{2}{3}; z \right), \quad z^{1/3} \,_2F_1 \left( \frac{2}{3} + i \frac{\sqrt{2}}{3}, \frac{2}{3} - i \frac{\sqrt{2}}{3}; \frac{4}{3}; z \right).
\end{align*}
This agrees with the solutions found by Schlemm.
\end{remark}

\begin{proof}
By the remark after Proposition \ref{prop:stationary_dist_functional_eqn} it follows that $\piN(\mx)$ has the form
\begin{align}\label{eqn:two_pi_form}
\piN(\mx) = \sum_{i \in \Z_N} q_N(\mx_i, \mx_{i+1})
\end{align}
for a yet to be determined function $q_N$. We also have the functional equation \eqref{eqn:stationary_dist_functional_eqn} for the stationary distribution, and by the shift invariance of the model we can write
\begin{align*}
\piN(\mL_i \my + \mL_i^c \mx) = \piN(\mx_{i+k}, \mx_{i+k+1}, \my_{i+k+2}, \ldots, \my_{i+k-1}).
\end{align*}
Combining this with \eqref{eqn:two_pi_form} shows that we can write $\piN(\mL_i \my + \mL_i^c \mx)$ as
\begin{align}\label{eqn:two_q_sum}
q_N(\my_{i+k-1}, \mx_{i+k}) + q_N(\mx_{i+k}, \mx_{i+k+1}) + q_N(\mx_{i+k+1}, \my_{i+k+2}) + \sum_{j=3}^{N-1} q_N(\my_{i+k+j-1}, \my_{i+k+j}).
\end{align}
Similarly the indicator function in the functional equation \eqref{eqn:stationary_dist_functional_eqn} can be written as
\begin{align}\label{eqn:two_indicator_function}
\1{a(\mL_i \my + \mL_i^c \mx) = i} = \1{\my_i \leq \min \{ \mx_{i+k}, \mx_{i+k+1} \}} \1{ \my_{j} \geq \my_i \textrm{ for } j \in \Z_N \backslash \{i, i+k, i+k+1\}}.
\end{align}
For all $\mx \in \mcU_N$ let $\mC_{N,i}(\mx)$ denote the set of all $\my \in \mcU_N$ satisfying the inequalities on the right hand side of \eqref{eqn:two_indicator_function}. To find the stationary distribution we insert \eqref{eqn:two_pi_form}, \eqref{eqn:two_q_sum}, \eqref{eqn:two_indicator_function} into \eqref{eqn:stationary_dist_functional_eqn} and obtain
\begin{align*}
&\sum_{i \in \Z_N} q_N(\mx_i, \mx_{i+1}) \\
\, &= \!\!\! \sum_{i \in \Z_N} \! \int_{\mC_{N,i}(\mx)} \!\!\!\! q_N(\my_{i+k-1}, \mx_{i+k}) + q_N(\mx_{i+k}, \mx_{i+k+1}) + q_N(\mx_{i+k+1}, \my_{i+k+2}) + \!\! \sum_{j=3}^{N-1} \! q_N(\my_{i+k+j-1}, \my_{i+k+j}) \, d \my
\end{align*}
Note that each integral on the right hand side above is a function of only $\mx_{i+k}$ and $\mx_{i+k+1}$ (all other variables get integrated out), and similarly each term on the left hand side is a function of two adjacent variables. Therefore we search for solutions of the form
\begin{align}\label{eqn:two_operator_eqn}
q_N(\mx_i, \mx_{i+1}) &= \int_{\mC_{N,i-k}(\mx)} \!\!\!\! q_N(\my_{i-1}, \mx_{i}) + q_N(\mx_{i}, \mx_{i+1}) + q_N(\mx_{i+1}, \my_{i+2}) + \!\! \sum_{j=3}^{N-1} \! q_N(\my_{i+j-1}, \my_{i+j}) \, d \my 
\end{align}
Since the equation above is the same for all $i$ it is enough to consider $i = 0$, and from now on we write $\mx_0 = u$ and $\mx_1 = v$. The range of integration becomes the set
\begin{align*}
\mC_{N}(u,v) := \{ \my \in \mcU_N : \my_{-k} \leq \min \{ u,v \}, \my_j \geq \my_{-k} \textrm{ for } j \in \Z_N \backslash \{-k, 0, 1 \} \}.
\end{align*} 
Note that $\my_{-k}$ plays a different role from the other $\my_j$ terms in equation \eqref{eqn:two_operator_eqn}, since they are separate in the set $\mC_{N}(u,v)$. Separating out the $\my_{-k}$ terms from equation \eqref{eqn:two_operator_eqn} leads to
\begin{align}
q_N(u,v) &= \int_{\mC_{N}(u,v)} \!\!\!\! q_N(\my_{-1}, u) + q_N(u, v) + q_N(v, \my_{2}) + q_N(\my_{-k-1}, \my_{-k}) + q_N(\my_{-k}, \my_{-k + 1}) \notag \\
 & \quad \quad \quad \quad \quad \quad \quad + \!\!\!\!\!\!\!\!\!\!\!\!\!\!\! \sum_{j \not \in \{ -1,0,1,-k-1,-k \}} \!\!\!\!\!\!\!\!\!\!\!\!\!\! q_N(\my_{j}, \my_{j+1}) \, d \my. \notag
\end{align}
Now $q_N$ is an unknown function on $[0,1]^2$ satisfying the above equation, with the constraints
\begin{align*}
q_N(u,v) \geq 0, \quad \int_{[0,1]^2} q_N(u,v) \, du \, dv = \frac{1}{N}
\end{align*}
that make it a density. If one inserts a symmetric function $f$ in place of $q_N$ into the right hand side of \eqref{eqn:two_operator_eqn} then the output is also a symmetric function, which motivates us to restrict our search to symmetric solutions $q_N(u,v) = q_N(v,u)$. With this in mind we restrict the domain of definition of the function $q_N$ to the set $\mathcal{S} = \{(u,v) : 0 \leq v \leq u \leq 1 \}$, and we symmetrize the right hand side to be defined on this domain (which, in general, means that we replace a function $f$ on $[0,1]^2$ with its symmetrization $\tilde{f}(u,v) = (f(u,v) + f(v,u))/2$ on $\mathcal{S}$). Then tedious but straightforward computations lead us to the equation
\begin{align}
q_N(u, v) &= a_{N-2}(v) q_N(u, v) + a_{N-3}(v) \int_u^1 q(s, u) \, ds + a_{N-3}(v) \int_v^1 q(s, v) \, ds \notag \\
& \,\,\, + \int_0^v a_{N-3}(t) q(u, t) \, dt + a_{N-3}(v) \int_v^u q(u, t) \, dt + \int_0^v a_{N-3}(t) q(v, t) \, dt \notag \\
& \,\,\, + 2 \int_0^v (1-t)^{N-4} \int_t^1 q(s, t) \, ds \, dt   + (N-5) \int_0^v (1-r)^{N-5} \int_r^1 \int_r^s q(s, r) \, dr \, ds \, dr \notag \\
& \,\,\, + (N-5) \int_0^v (1-r)^{N-5} \int_r^1 \int_r^1 q(s, t) \, ds \, dt \, dr, \label{eqn:two_symmetric_eqn}
\end{align}
Note that the $a_k(v)$ terms are simply the probabilities
\begin{align*}
a_k(v) = \P \left( U_0 \leq v, \, U_{j} \geq U_0 \textrm{ for } j = 1,2,\ldots,k-1 \right)
\end{align*}
where the $U_i$ are independent uniform random variables on $[0,1]$. The $a_{N-2}$ and $a_{N-3}$ factors that appear in \eqref{eqn:two_symmetric_eqn} are a result of integrating out most components of the $\my$ variables in \eqref{eqn:two_operator_eqn} over the sets $\mC_{N,i}(\mx)$.

Thus we are left to solve \eqref{eqn:two_symmetric_eqn}, and at this point we introduce the ansatz $q_N(u,v) = B_N(u)G_N(v) + A_N(v)$ for functions $B_N, G_N, A_N : [0,1] \to \R$. This is the form that already appears in Schlemm's solution \cite{schlemm:five_species}, and which we came upon by using a forward iteration method on \eqref{eqn:two_symmetric_eqn} and observing that this form is invariant under the iteration. Inserting it into \eqref{eqn:two_symmetric_eqn} results in the equation
\begin{align}
B_N(u) G_N(v) + A_N(v) &= a_{N-2}(v) B_N(u) G_N(v) + a_{N-2}(v) A_N(v) \notag \\
&\,\, + a_{N-3}(v) G_N(u) \int_u^1 B_N(s) \, ds + a_{N-3}(v) (1-u)A_N(u) \notag \\
&\,\, + a_{N-3}(v) G_N(v) \int_v^1 B_N(s) \, ds + a_{N-3}(v)(1-v)A_N(v) \notag \\
&\,\, + B_N(u) \int_0^v a_{N-3}(t) G_N(t) \, dt + \int_0^v a_{N-3}(t) A_N(t) \, dt \notag \\
&\,\, + a_{N-3}(v) B_N(u) \int_v^u G_N(t) \,dt + a_{N-3}(v) \int_v^u A_N(t) \, dt \label{eqn: insertAnsatz} \\
&\,\, + B_N(v) \int_0^v a_{N-3}(t) G_N(t) \,dt + \int_0^v a_{N-3}(t) A_N(t) \, dt \notag \\
&\,\, + 2 \int_0^v (1-t)^{N-4}G_N(t) \int_t^1 B_N(s) \,ds \, dt + 2 \int_0^v (1-t)^{N-3} A_N(t) \, dt \notag \\
& \,\, + (N-5) \int_0^v (1-r)^{N-5} \int_r^1 \int_r^s q(s, t) \, dt \, ds \, dr  \notag \\
& \,\, + (N-5) \int_0^v (1-r)^{N-5} \int_r^1 \int_t^1 q(s, t) \, ds \, dt \, dr. \notag
\end{align}
Equation \eqref{eqn:thirdEqn} is found by equating all the terms in \eqref{eqn: insertAnsatz} involving $v$. We then differentiate \eqref{eqn: insertAnsatz} with respect to $u$ to get
\begin{align*}
A_N'(u) &= \frac{1}{1-u} \left[ - G_N'(u) \int_u^1 B_N(s) \, ds - B_N'(u) \int_0^u G_N(s) ds \right] \\
&\,\,  + \frac{1}{a_{N-3}(v)(1-u)} \left[ (1 - a_{N-2}(v))B_N'(u)G(v) + a_{N-3}(v)B_N'(u) \int_0^v G_N(s) \, ds \right. \\
&\quad\quad\quad\quad\quad\quad\quad\quad\quad\quad \left. - B_N'(u) \int_0^v a_{N-3}(s) G_N(s) \, ds \right].
\end{align*}
By equating all terms that are functions of $u$ alone we derive \eqref{eqn:Afunction}. This leaves the terms involving both $u$ and $v$
\begin{align*}
0 =  (1 - a_{N-2}(v))B_N'(u)G(v) + a_{N-3}(v)B_N'(u) \int_0^v G_N(s) \, ds  - B_N'(u) \int_0^v a_{N-3}(s) G_N(s) \, ds.
\end{align*}
We factor out $B_N'(u)$ and then differentiate once with respect to $v$ to derive \eqref{eqn:DEforG}.

\end{proof}

\bibliographystyle{alpha}
\bibliography{../BibTex/citations}

\end{document}